\documentclass[12pt,a4paper]{amsart}
\usepackage[utf8]{inputenc}
\usepackage[english]{babel}
\usepackage{amsmath}
\usepackage{amsfonts}
\usepackage{amssymb}
\usepackage{amsthm}
\usepackage{mathrsfs} 
\usepackage{wasysym}
\usepackage{verbatim} 
\usepackage{hyperref}
\usepackage{graphicx}
\usepackage[arrow, matrix, curve]{xy}

\newcommand{\ra}{\rightarrow}

\newcommand{\sra}{\twoheadrightarrow}
\newcommand{\hra}{\hookrightarrow}

\newcommand{\sira}{\stackrel{\sim}{\rightarrow}}
\newcommand{\sirra}{\stackrel{\approx}{\rightarrow}}

\newcommand{\ua}{\uparrow}
\newcommand{\RA}{\Rightarrow}

\newcommand{\Bl}[1]{{\mathbb{#1}}}

\newcommand{\DZ}{\Bl{Z}}
\newcommand{\DN}{\Bl{N}}

\newcommand{\op}{\operatorname}

\newcommand {\td} {\tilde}
\newcommand {\wt} {\tilde}
\newcommand {\wh} {\hat}

\newcommand  {\Idp}[1]  {{1}_{#1}} 

\newcommand{\opU}{\mathrm{U}}

\newcommand{\opS}{\mathrm{S}}

\newcommand  {\ZZ}{\ensuremath{{\mathbb Z}}}

\newcommand{\calA}{\mathcal{A}}
\newcommand{\calB}{\mathcal{B}}
\newcommand{\calC}{\mathcal{C}}

\newcommand  {\frb}{\ensuremath{{\mathfrak b}}}

\newcommand  {\frg}{\ensuremath{{\mathfrak g}}}
\newcommand  {\frh}{\ensuremath{{\mathfrak h}}}

\newcommand {\scO}{\ensuremath{{\mathcal{O}}}}

\newcommand {\eps}{\ensuremath{{\epsilon}}}

\theoremstyle{plain} 

\newtheorem{thm}[subsection]{Theorem}
\newtheorem{lemma}[subsection]{Lemma}
\newtheorem{prop}[subsection]{Proposition}
\newtheorem{cor}[subsection]{Corollary}

\theoremstyle{definition} 

\newtheorem{defi}[subsection]{Definition}
\theoremstyle{remark} 
\newtheorem{example}[subsection]{Example}
\newtheorem{remark}[subsection]{Remark}
\newtheorem{remarkl}[subsection]{ }

 \usepackage{color}
\newcommand{\cemph}{\textcolor{black}} 
 \usepackage{alltt}
 \usepackage[active]{srcltx}
\begin{document}
\author{Michael Rottmaier and Wolfgang Soergel}
\title[Unicity of graded covers
 of the Category $\scO$]{Unicity of graded covers
 of the Category $\scO$ of Bernstein-Gelfand-Gelfand}
\address{Universit\"at Freiburg\\Mathematisches Institut
\\Eckerstra\ss e 1\\ D-79104 Freiburg\\Germany}
\email{wolfgang.soergel@math.uni-freiburg.de, Kichel@gmx.net}
\thanks{partially supported by the DFG in the framework of
SPP 1388 and GK 1821}
  \begin{abstract}
    We show that the standard graded cover of the well-known category $\scO$
    of Bernstein-Gelfand-Gelfand can be characterized by its compatibility with
    the action of the center of the enveloping algebra.
  \end{abstract}

\maketitle
\tableofcontents

\section{Introduction}

Let $ \frg \supset \frb \supset \frh$ be a complex
semi-simple Lie-algebra 
with a choice of a Borel and Cartan subalgebra. 
In \cite{BGGO} Bernstein, Gelfand
and Gelfand introduced  the so called category 
$\scO=\scO(\frg,\frb)$ of representations
of $\frg$. 
{Later on Beilinson and Ginzburg \cite{BGi} argued}
that it is natural to study $\ZZ$-gradings
of category $\scO$, {see also
\cite{BGSo}}. 
In this article we  introduce the notion of a graded cover,
a  generalization of the notion
of a $\ZZ$-grading which seemed to be more natural to us, and prove the 
following {uniqueness theorem for graded covers of $\mathcal O$},
 to 
be explained in  more detail in the later parts of this introduction.
 
\begin{thm}[{\bf Uniqueness of graded covers of $\scO$}] 
  \begin{enumerate}
  \item Cate\-gory $\scO$
admits a graded cover compatible with
    the action of the center of the universal enveloping algebra of $\frg$;
    \item If two graded covers of $\scO$ are 
both compatible with the action of
    the center of the universal enveloping algebra, they are cover-equivalent.
  \end{enumerate}
 \label{thm:MainTheorem}\end{thm}
\begin{remarkl}
  Graded covers of category
 $\scO$ which are compatible with the
  action of the center have already been constructed in \cite{So}
  and \cite{BGSo}. The main point of this article is to show that compatibility
 with the action of the center already determines a graded cover of
  $\scO$ up to cover-equivalence. 
To our knowledge this statement is already  new
in the case of
 $\ZZ$-gradings \cite[4.3]{BGSo}, although to see this case treated
the reader 
could skip most of the paper and 
go directly to the proof of \ref{thm: GradingsOnO}, which mainly reduces to
a careful application of the bicentralizer property as discussed in section
\ref{gbz}.
In the body of the paper, we mainly investigate the notion of a 
graded cover.  
\end{remarkl}

\begin{remarkl}
  An analogous theorem holds with the same proof for 
the modular  versions of category $\mathcal O$ considered in \cite{So-R,
  RSW}. It is for this case, that the generalization from gradings to
graded covers is really needed.  An analogous theorem also holds for
the category of all Harish-Chandra modules over a complex connected reductive 
algebraic group, considered as a real Lie group. 
If we
interpret our Harish-Chandra modules as bimodules over the enveloping algebra
in the usual way and
 restrict to the subcategories
of objects killed from the left by a fixed power of a 
given maximal ideal in the center of the enveloping
algebra, the same proof in conjunction with \cite{HCH} will work. 
To deduce the general case,  some additional arguments are needed to 
justify passing to
the limit, which can be found in the diploma thesis of M. Rottmaier. 
\end{remarkl}

\begin{defi}
  An abelian category in which each object has finite length will
    henceforth be called an {\bf artinian category}.
\end{defi}

\begin{defi}
  By a {\bf graded cover} of an artinian category $\mathcal A$ 
we understand a triple\label{def:zgradedversion} 
  $(\wt{\mathcal A}, v, \eps)$ consisting of an abelian category $\wt{\mathcal A}$
  equipped with a strict automorphism $[1]$ ``shift the grading'', an exact
  functor $v: \wt{\mathcal A} \rightarrow \mathcal A$ ``forget the grading''
and an isotransformation of functors
  $\eps: v \overset{\sim}{\Rightarrow} v[1]$, such that the following hold:
  \begin{enumerate}
  \item For all $ M, N \in \wt{\mathcal A}$ the pair $(v, \eps)$ induces an
    isomorphism on the homomorphism groups $\bigoplus_{i \in \mathbb{Z}}
    \wt{\mathcal A}(M, N[i]) \overset{\sim}{\rightarrow} \mathcal A(vM, vN)$;
  \item Given $ M \in \mathcal A $,
$N\in\cemph{\tilde{\mathcal A}}$ and an 
epimorphism $M\sra vN$ there exists $P\in \cemph{\tilde{\mathcal A}}$ and
a\cemph{n epi}morphism $vP\sra M$ such 
that the composition $vP\ra vN$ comes from a
morphism $P\ra N$ in $\tilde{\mathcal A}$.
  \end{enumerate}
\end{defi}

\begin{remark}
  The main difference to the concept of a $\ZZ$-grading
in the sense of \cite[4.3]{BGSo} is that for our graded covers 
we don't ask for any kind of  positivity or
semisimplicity. In particular, if we start with a grading
and  ``change all degrees to their negatives'', we
would always  get a graded cover again, but most of the time this 
would not be
a grading anymore. If $A$ is an left-artinian ring with a 
$\DZ$-grading, for which in \cemph{this} article 
we ask  no positivity condition whatsoever, then 
the forgetting of the grading on finitely generated graded modules 
$v: \tilde A\op{-Modf}^\DZ\ra A\op{-Modf}$
 with the obvious $\epsilon$ always 
is a graded cover, see \ref{ex:StandardExample}.
In \ref{ACO} we will show that the opposed category of a graded cover 
is a graded cover of the opposed category. 
\end{remark}

\begin{remark}
We would like to know whether  condition (2) 
will follow in general, when we ask it only in case $N=0$. 
\end{remark}

\begin{remark}
 We will try to strictly follow a notation, where calligraphic letters
$\mathcal A,\mathcal B,\ldots$ denote categories, 
roman capitals $F,G,\ldots$ denote functors between 
or objects of our categories, and
little greek letters $\tau, \epsilon,\ldots$ denote transformations. 
The only exception is the ``forgetting of grading'' in all its variants, 
to be denoted by the small letter $v$
 although it is a functor.
\end{remark}

\begin{defi}
  We say that a graded cover $(\wt{\scO}, v, \eps)$
of the BGG-category $\scO$ is {\bf compatible with the action of the center} 
  $Z\subset \opU(\frg)$ 
of the enveloping algebra iff the following holds: Given an object 
$\td{M} \in \wt{\scO}$
  and a maximal ideal  $\chi\subset Z$
 such that $\chi^{n}(v\td{M}) = 0$ for some
  $n \in\DN$, the induced morphism
  \[ Z / \chi^{n} \rightarrow \text{End}_{\frg}(v\td{M}), \quad z +
  \chi^{n} \mapsto (z \cdot) \] is homogeneous for the grading on
  End$_{\frg}(v\td{M})$ induced by the pair $(v, \eps)$ and the natural
  grading on $Z/\chi^{n}$ induced by the Harish-Chandra-homomorphism 
 as explained in the  next remark \ref{par:natgrading}.
\end{defi}
\begin{remarkl}[{\bf The natural grading on
      $Z/\chi^n$}] \label{par:natgrading} Let $S=\opS(\mathfrak h)$ be the
    symmetric algebra of our Cartan.  The Weyl group $W$ acts on it in a
    natural way. We
    have the Harish-Chandra isomorphism $Z\overset{\sim}{\rightarrow} S^W$.
    For any maximal ideal $\lambda\subset S$ let $W_\lambda$ be its isotropy
    group and $Y=Y(\lambda)\subset S$ be the $W_\lambda$-invariants and put
    $\chi=\lambda\cap Z$ and $\mu=\lambda\cap Y$.  Now general results in
    invariant theory \cite[Ex. 3.18.]{Liu} tell us, that the natural maps 
are in fact
    isomorphisms
$$Z^\wedge_\chi\overset{\sim}{\rightarrow} (S^\wedge_\lambda)^{W_\lambda}
\overset{\sim}{\leftarrow}Y^\wedge_\mu$$ from the completion of the invariants
to the invariants of the completion, leading to isomorphisms
$Z/\chi^n\overset{\sim}{\rightarrow} Y/\mu^n$.  Moving the obvious
$W$-invariant grading on $S$ with the comorphism of the translation by
$\lambda$, we obtain a $W_\lambda$-invariant 
grading on $S$, which induces a grading on
$Y$ with $\mu=Y^{>0}$ its part of positive degree. This way we get a natural
grading on $Y/\mu^n$. The reader may easily 
check that the induced grading on $Z/\chi^n$
doesn't depend on the choice of $\lambda$. We call
it the {\bf natural grading on $Z/\chi^n$}.
\end{remarkl}

\begin{defi}
  Let $(\wt{\mathcal A},\td v, \td \eps)$ and $(\wh{\mathcal A}, \hat{v},
  \hat{\eps})$ be\label{qez} 
  graded covers of an artinian category $\mathcal A$. 
An {\bf cover-equivalence} is a triple
  $(F, \pi,\epsilon)$, where $F: \wt{\mathcal A} \rightarrow \wh{\mathcal A}$ is an additive functor
  and $\epsilon :[1]F \overset{\sim}{\Rightarrow} F[1]$ and $\pi: \hat{v}  F
  \overset{\sim}{\Rightarrow} \td v$ are isotransformations of functors such that
  the following diagram of isotransformations of functors commutes:
  \[
  \begin{xy}
    \xymatrix{	\hat{v} [1] F \ar@{=>}[d]^\wr_{\hat{\epsilon}}
      \ar@{=>}[r]^{\epsilon}_{\sim}
      &	\hat{v} F [1] \ar@{=>} [r]^{\pi}_{\sim} 		& \td v [1]
      \ar@{=>}[d]^{\td \epsilon}_\wr  \\
      \hat{v} F \ar@{=>}[rr]^{\pi}_{\sim} & & \td v }
  \end{xy}
  \]
Two covers are said to be {\bf cover-equivalent} iff there is a
cover-equi\-valence from one to the other. We will show in 
\ref{coeq} that this is indeed an equivalence relation. 
 \end{defi}

 \begin{remarkl}
   This generalizes  the definition of equivalence of gradings given in
   \cite[4.3.1.2]{BGSo}. We will show in \ref{grCO} that
given a left-artinian ring $A$  every graded cover of
   $A\op{-Modf}$ is 
cover-equivalent to the graded cover given by a $\DZ$-grading 
on $A$. The question when
   two graded 
covers of this type are cover-equivalent to each other is discussed in
   \ref{prop:EquivalenceCriterium}.
 \end{remarkl}

\section{Graded covers of artinian categories}

\begin{example}  \label{ex:StandardExample}

Let $A$ be a left-artinian ring. Given any $\ZZ$-grading $\, \wt{} \,$ on $A$ and let 
($\wt{A}\operatorname{-Modf}^{\ZZ}$, $v$, $\eps$) be the category of finitely generated $\ZZ$-graded left $\wt{A}$-modules with morphisms homogeneous of degree 0 and $(v,\eps)$ the natural forgetting of the grading.
Then ($\wt{A}\operatorname{-Modf}^{\ZZ}$, $v$, $\eps$) is a graded cover of $A$-Modf.   
To check the second condition in definition \ref{def:zgradedversion}, 
 let $M\sra N$ be a surjection of a not
necessarily graded module onto a graded module. A generating system of $M$
gives
a generating system of $N$. Take the nonzero  homogeneous components
of its elements. It is possible to choose preimages  of these 
components in $M$ in such a way, that they generate $M$. Then a suitable 
graded free $A$-module $P$ with its basis vectors going to these 
preimages will do the job. 
\end{example}

\begin{remark}
Every graded cover  of an artinian category
is also artinian, as the length can get only bigger when we apply
an exact functor which doesn't annihilate any object.
\end{remark}

\begin{defi}

Given a graded cover $(\wt{\calA}, v, \eps)$ of an artinian category 
$\calA$, a
\textbf{$\ZZ$-graded lift} or for short {\bf lift} of an object $M \in \calA$
is a pair $(\td{M}, \varphi)$ with $\td{M} \in \wt{\calA}$ and 
$\varphi: v \td{M} \overset{\sim}{\rightarrow} M$ an isomorphism in $\calA$.
\end{defi}
\begin{lemma} \label{lem:UniquenessOfLiftOfIndemcomp}

If for an indecomposable object there exists a lift, then this 
lift is unique up to isomorphism and shift. 

\end{lemma}

\begin{proof}
Given an indecomposable object $M \in \calA$, its endomorphism ring $\calA(M,M)$
is a local ring. Suppose there are two  lifts ($\td{M}, \varphi)$, $(\hat{M}, \psi)$ of $M$. 
We can then decompose the identity morphism of $M$
into homogeneous components. Because the non-units in $\calA(M,M)$ form its maximal ideal,
at least one of the homogeneous components has to be a unit, i.e. an isomorphism. 
\end{proof}

\begin{lemma} \label{lem:ShiftedsAreNotIso}

Given a graded cover $(\wt{\calA}, v, \eps)$ of an artinian category
$\calA$ a non-trivial object $0\neq M \in \wt{\calA}$ is never isomorphic to
its shifted
versions $M[i]$ for $i \neq 0$.

\end{lemma}

\begin{proof}

It is enough to prove the statement for simple objects.  
From now on let $M \in \wt{A}$ be  simple. 
Suppose there is an isomorphism $M \overset{\sim}{\rightarrow} M[i]$ for some 
 $i \neq 0$.
Then the endomorphism ring of $vM \in \calA$ is given by twisted Laurent
series over 
a skew-field, more precisely 
$\calA(vM, vM)$ is of the form $K^{\sigma}[X, X^{-1}]$
where the skew-field $K = \wt{\calA}(M, M)$ is the endomorphism
ring of $M$ in $\wt{\calA}$, $\sigma: K \rightarrow K$ an automorphism of skew-fields and $cX = X\sigma(c)$ for all $c \in K$.
Obviously  $0$ and $1$ are 
the only idempotents in $K^{\sigma}[X,X^{-1}]$, so $vM \in \calA$
is an indecomposable object. 
On the other hand it has finite length by assumption. 
Thus, by a version of Fittings lemma, all elements in the endomorphism 
ring $\calA(vM, vM)$ have to be either units or nilpotent,
and this is just not the case. 
\end{proof}

\begin{lemma} \label{lem:SOL}
Given a graded cover $(\wt{\calA}, v, \eps)$ of an artinian category $\calA$, forgetting of the grading induces a bijection of sets
\[ (\mathrm{irr} \wt{\calA}) / \ZZ \overset{\sim}{\rightarrow} \mathrm{irr} \calA \]
between the isomorphism classes of simple objects in the graded cover modulo shift
and the isomorphism classes of simple objects in $\calA$. 

\end{lemma}

\begin{proof}

First we show that each epimorphism $vM \twoheadrightarrow L$ in $\calA$,
where $M \in \wt{\calA}$ 
{is simple} 
and $L \in \calA$ {is non-zero}, 
has to be an isomorphism.
By definition of a graded cover
 we find an object $N \in \wt{\calA}$ such that $vN$ maps epimorphically onto ker$(vM \twoheadrightarrow L)$, i.e. there is 
a resolution $vN \overset{\varphi}{\rightarrow} vM \twoheadrightarrow L$ of $L$ in $\calA$ by 
objects having a $\ZZ$-graded lift.
We can decompose $\varphi$ into homogeneous components $ \varphi = \sum_{i \in
  \ZZ}\varphi_{i}$ and each summand $\varphi_{{i}}: N \rightarrow M[i]$
has to be either trivial or an epimorphism in $\wt{\calA}$. 
{Using \ref{lem:ShiftedsAreNotIso}}, 
the non-trivial summands give an epimorphism 
$$(\varphi_{i}): N \twoheadrightarrow \bigoplus_{i \in \ZZ, \;\varphi_{i} \neq 0}M[i]$$
in $\wt{\calA}$.
{This has to stay an epimorphism when we forget the grading
and postcompose with the morphism 
$\bigoplus_{\varphi_{i} \neq 0}vM\twoheadrightarrow vM$ adding up the components.
So if $\varphi$ is non-zero, it is a surjection. Thus}
we have shown that the forgetting of the grading induces a map $(\mathrm{irr} \wt{\calA}) / 
\ZZ \rightarrow \mathrm{irr} \calA $.
To see that it is surjective, take $L \in \mathrm{irr}\calA$.
By our assumptions   there is an $M \in \wt{\calA}$ together with an
epimorphism 
$vM \twoheadrightarrow L$; consider the set of all objects in $\wt{\calA}$
which map, after forgetting the grading, epimorphically onto $L$ and choose 
among  them an object $M \in \wt{\calA}$ 
of minimal length. 
If $M$ is not simple, there is a non-trivial sub-object $K \subset M$
with non-trivial quotient and we get a short exact sequence
\[ 0 \rightarrow vK \rightarrow vM \rightarrow vC \rightarrow 0.\]
Then the restriction of $vM \twoheadrightarrow L$ to $vK$ must be trivial, otherwise it  contradicts the
minimal length assumption on $M$; therefore $vC$ has to map epimorphically onto $L$ 
again contradicting the minimal length assumption on $M$. We conclude that $M \in \wt{A}$ was already simple.    
The proof of injectivity is left to the reader. 
\end{proof}
\begin{prop} \label{prop:ProjectivesLift}

Projective  objects  do lift.

\end{prop}

\begin{remark}
Using the stability of graded covers by passing to the opposed categories 
\ref{ACO}, we easily deduce that injective objects do lift as well.   
\end{remark}
\begin{proof}
It is enough to prove the statement for indecomposable projective objects.  
 Let $P$ be one of those.
 It is known that $P$ admits a unique simple quotient, 
which in turn by \ref{lem:SOL} admits a graded lift, so that we can write 
\[P \twoheadrightarrow vL \]  
with $L\in\tilde{A}$  a simple object. 
By our definition of a graded cover we can find an epimorphism
$vM\sra P$ such that the composition $vM\sra P\sra vL$ 
comes from a morphism $M\ra L$. 
Assume now in addition, that $M$ has 
minimal length for such a situation. 
If we can show that $v{M}$ is indecomposable we are done, 
because $P$ is projective and
 thus the morphism $v{M} \twoheadrightarrow P$ splits. 
Suppose $v{M} \cong A \oplus B$. Then one summand, say $A$,
 has to map epimorphically onto $L$. If $B$ is not zero, 
then $B$ also has a simple quotient $\pi: B \twoheadrightarrow vE$
and  we get an epimorphism  $\psi: v{M} \twoheadrightarrow vL \oplus vE$. 
We can
decompose the composition $$\lambda=\operatorname{pr}_{2}\circ \psi: v {M} 
\rightarrow vL \oplus vE \twoheadrightarrow vE $$ into homogeneous components
$\sum_{i \in \ZZ} \lambda_{i}$. 
If there was a non-zero $\lambda_i:  {M}\rightarrow {E}[i]$ 
with ${E}[i]\not\cong {L}$, then
$\ker \lambda_i$ would also surject onto ${L}$
and $v(\ker \lambda_i)$ would  surject onto $v{L}$ and thus onto
$P$, contradicting our assumption of minimal length. 
So we may assume our epimorphism $\psi$ is obtained by forgetting the grading 
from an epimorphism 
$(\td\lambda,\td\varphi): M\rightarrow  L\oplus L$. 
But then again $v(\ker\td\lambda)$ will still surject onto $vL$,
contradicting our assumption of minimal length. 
Thus $v M$ is indecomposable and the split epimorphism 
$v M\twoheadrightarrow P$ has to be an isomorphism. 
\end{proof}

\begin{cor} \label{cor:IndecProjLift}
Let $(\wt{\calA}, v, \eps)$ be a graded cover  of an artinian category $\calA$ and
suppose $\calA$ has enough projective objects. Then
forgetting  the grading induces a bijection of sets
\[ (\mathrm{inProj} \wt{\calA}) / \ZZ \overset{\sim}{\rightarrow} \mathrm{inProj} \calA \]
between the isomorphism classes of indecomposable projective objects in the graded cover modulo shift
and the isomorphism classes of indecomposable projective objects in $\calA$. 

\end{cor}

\begin{proof}
It is well known that if an artinian category $\calA$ has enough projectives, 
taking the projective cover
gives a bijection between the set of isomorphism classes of simple 
objects in $\calA$ and the set of isomorphism classes of indecomposable
projective objects in $\calA$. 
So both sides are in bijection to the corresponding 
sets of isomorphism classes of simple objects.
Since each projective admits a lift by  \ref{prop:ProjectivesLift},
and since such a lift clearly is again projective,
the statement follows from the existence and unicity statement
about lifts of simple objects \ref{lem:SOL}.  
\end{proof}
\begin{cor} \label{ex:ClassifyingLiftsOfA}
Let $A$ be a left-artinian ring  with a $\ZZ$-grading. Then:
\begin{enumerate}
\item There exists a complete system 
of primitive pairwise orthogonal
  idempotents in $A$ such that  all its elements  are
homogeneous;
\item If $(1_x)_{x\in I}$ is such a system,
$(\wt{M}, \varphi)$ a graded lift  of  $A$ considered as a
  left $A$-module, and 
$\tilde{A}$ the lift of $A$ 
given by the $\ZZ$-grading, then there exists a map $n:I\ra\DZ$ 
along with an isomorphism of graded left $\td A$-modules
  $$\wt{M} \overset{\sim}{\rightarrow} \bigoplus_{x \in
    I}\tilde{A}\Idp{x}[n(x)]$$ 
\end{enumerate}
\end{cor}

\begin{proof}
It is easy to see that every homomorphism
$\td A\rightarrow \td A[i]$ of graded left $A$-Modules is the 
multiplication from the right with an element of $A$ homogeneous of degree $i$. 
There is a direct sum decomposition $\tilde{A} \cong \bigoplus_{x \in I} P_x$
into
indecomposable objects in $\wt{A}\operatorname{-Modf}^{\ZZ}$, and its summands are
projective. 
The corresponding  idempotent endomorphisms of $\td A$ are 
right multiplications with some idempotents $1_x\in A$, homogeneous of degree
zero. Forgetting the grading on the $P_x$ we get 
indecomposable projective $A$-modules by \ref{cor:IndecProjLift},
and thus our family $( \Idp{x} )_{x \in I}$ is
a full set of primitive orthogonal idempotents in $A$.
For the second statement
let $\td M=\bigoplus_{y\in J} Q_y$ be a direct sum decomposition
into
indecomposable objects in $\wt{A}\operatorname{-Modf}^{\ZZ}$. Again its summands are
projective, so by  \ref{cor:IndecProjLift}
they stay indecomposable when we forget the grading. 
So there is a bijection $\sigma:I\overset{\sim}{\rightarrow} J$ with
$vP_x\cong vQ_{\sigma(x)}$ and by the uniqueness of lifts
\ref{lem:UniquenessOfLiftOfIndemcomp} we find  $P_x[n(x)]\cong Q_{\sigma(x)}$ 
for suitable $n(x)\in \ZZ$. 
\end{proof}
\section{Alternative definition of graded covers}
\begin{remarkl}
  The following proposition establishes the relation to the concept of a
  $\DZ$-grading as  introduced in \cite{BGSo}. It also ensures our concept
of graded cover to be stable upon passing to the opposed categories.
Apart from that, this section is not relevant for the rest of this article. 
\end{remarkl}

\begin{prop}
  Let $\mathcal A$ be an artinian category. 
A triple\label{ACO} 
  $(\wt{\mathcal A}, v, \eps)$ consisting of an abelian category $\wt{\mathcal A}$
  equipped with a strict automorphism $[1]$, an exact
  functor $v: \wt{\mathcal A} \rightarrow \mathcal A$ and an isotransformation of functors
  $\eps: v \overset{\sim}{\Rightarrow} v[1]$ is a graded cover of
$\mathcal A$ if and only if the following hold:
  \begin{enumerate}
  \item For all $ M, N \in \wt{\mathcal A}$ the pair $(v, \eps)$ induces 
    isomorphisms of extensions $\bigoplus_{i \in \mathbb{Z}}
    \op{Ext}^j_{\wt{\mathcal A}}(M, N[i]) \overset{\sim}{\rightarrow} \op{Ext}^j_{{\mathcal A}}(vM, vN)$;
  \item Every irreducible object in $  \mathcal A $  admits a graded lift.
  \end{enumerate}
\end{prop}

\begin{proof}
  Let us first show that every graded cover has these two properties.
For the second property, this follows from \ref{lem:SOL}.
For the first property, we use
$$\op{Ext}^j_{{\mathcal A}}(vM,vN)=
\varinjlim_Q\op{Hot}^j_{{\mathcal A}}(Q,vN)$$
where $Q$ runs over the system of all resolutions $Q\ra vM$ of $vM$. 
Our condition (2) on a graded cover ensures that if we take all
resolutions $P\ra M$ in $\tilde{\mathcal A}$, the resolutions 
$vP\ra vM$ will be cofinal in the system of all resolutions of $vM$ and
thus give the same limit. The claim follows.
Now let us show to the contrary that our two properties ensure both
conditions of the definition of a graded cover \ref{def:zgradedversion}.
The first condition is obvious. To show the second condition, we 
may use pull-backs and induction to restrict to the case
$\op{ker}(M\sra vN)$ is simple. Then by assumption
this kernel admits a graded lift, so we arrive at a short exact sequence
$$vL\hra M\sra vN$$
Now use the isomorphism
$\bigoplus_{i \in \mathbb{Z}}
    \op{Ext}^1_{\wt{\mathcal A}}(N, L[i]) \overset{\sim}{\rightarrow}
    \op{Ext}^1_{{\mathcal A}}(vN, vL)$
to write the class $e$ of the above extension as a finite sum of homogeneous
components
$e=\sum_{i=a}^b e_i$. We then get a commutative diagram
$$\begin{xy}
\xymatrix{\bigoplus_{i=a}^b L[i] \ar@^{(->}[r] \ar@^{=}[d]           &
  \bigoplus_{i=a}^b E_i \ar@{->>}[r]      & \bigoplus_{i=a}^b N  \\
          \bigoplus_{i=a}^b L[i]  \ar@^{(->}[r]          & P
          \ar@{->>}[r] \ar[u]     & N  \ar[u]_{\Delta}}
\end{xy}$$
with a right pullback square and by forgetting the grading another
  commutative diagram
$$\begin{xy}
\xymatrix{ \bigoplus_{i=a}^b vL  \ar@^{(->}[r] \ar[d]^{\sum}          & vP
          \ar@{->>}[r] \ar[d]     & vN \ar@^{=}[d] \\
          vL   \ar@^{(->}[r]            
&              M     \ar@{->>}[r]       &  vN
          }
\end{xy}
$$
with a left pushout square. This finishes the proof. 
\end{proof}

\section{Lifting functors to $\DZ$-graded covers} 

\begin{defi}
 A  {\bf  $\ZZ$-category} is a category  $\wt{\calA}$ 
together with a strict 
autoequivalence $[1]$ which we call ``shift of grading''.  
A \textbf{$\ZZ$-Functor} between $\ZZ$-categories is a 
pair $(F,\epsilon)$ consisting of a functor $F$ between the
  underlying categories and an isotransformation  $\epsilon:F[1] 
\stackrel{\sim}{\RA} [1]F$. 
\end{defi}


\begin{defi} \label{def:ZLiftOfFunctor}

Let $\calA$ and $\calB$ be  artinian categories, $F: \calA \rightarrow \calB$
an additive functor
 and let
$(\wt{\calA}, v, \eps)$ and $(\wt{\calB},  w , \eta)$ be graded covers of
$\calA$ and $\calB$
respectively. 
 A $\ZZ$\textbf{-graded lift of $F$} is a triple $(\wt{F},  \pi,\epsilon)$, 
where 
 $\tilde F:\tilde \calA \rightarrow \tilde \calB$ 
is an additive $\ZZ$-functor and 
 $ \pi:  w   \wt{F} \overset{\sim}{\Rightarrow} F  v$ 
and $\epsilon :[1] \tilde{F} \overset{\sim}{\Rightarrow} 
 \tilde{F}[1] $ are isotransformations of
 functors, such that the following diagram of isotransformations of functors commutes:
 \[
  \begin{xy}
    \xymatrix{	w [1] \tilde{F}  \ar@{=>}[d]^\wr_{\eta}
      \ar@{=>}[r]^\epsilon_{\sim}
	&
      w \tilde{F} [1] \ar@{=>} [r]^{ \pi}_{\sim} 		& F v [1] \ar@{=>}[d]^\epsilon_\wr  \\
      w \tilde{F} \ar@{=>}[rr]^{ \pi}_{\sim} & & F v }
  \end{xy}
  \]
\end{defi}

\begin{remarkl}
  By definition, a cover-equivalence as defined in \ref{qez} 
between two graded covers of a given
  artinian category 
is a 
graded lift of the identity functor in the sense of \ref{def:ZLiftOfFunctor}.
\end{remarkl}
\begin{remarkl}
  Take two $\ZZ$-graded left-artinian rings $\wt{A}$ and
  $\wt{B}$ and in addition a $B$-$A$-bimodule $X$ of finite length as left
  $B$-module.\label{ggr} Then obviously 
the functor $F=X\otimes_A: A\op{-Modf}\ra
  B\op{-Modf}$ admits a graded lift $\tilde F: \tilde A\op{-Modf}^\DZ\ra
  \tilde B\op{-Modf}^\DZ$ if and only if $X$ admits a $\DZ$-grading making it
  into a graded $\wt{B}$-$\wt{A}$-bimodule $\tilde X$. 
\end{remarkl}

\section{Comparing graded covers of module categories}

\begin{prop}
  Let $A$ be a left artinian ring
and $(\tilde{\mathcal A},v,\epsilon)$ a graded cover of $A\op{-Modf}$.
Then there exists a $\DZ$-grading   $\, \wt{} \,$ on $A$
such that $\tilde A\op{-Modf}^\DZ$ is cover-equivalent to 
$\mathcal A$.\label{grCO}  
\end{prop}

\begin{proof}
  By \ref{prop:ProjectivesLift} there exists a lift $(\tilde M,\varphi)$
of $A$ in $\tilde{\mathcal A}$. By assumption 
we obtain  isomorphisms
$$\bigoplus_i \tilde{\mathcal A}(\tilde M,\tilde M[i])\;\sira \; 
\op{End}_A(A)\;\stackrel{\sim}{\leftarrow} \; A^{\op{opp}}$$
Here the left map comes from forgetting the grading and the
right map from the action by right multiplication.  We leave it to the reader
to check that this grading on $A$ will do the job. 
\end{proof}

\begin{prop} \label{prop:EquivalenceCriterium}

Let $A$ be  a left-artinian ring and let $\, \wt{} \,$ and $\, \hat{} \,$ be two 
$\ZZ$-gradings on $A$. Then the following statements are equivalent: 

\begin{enumerate}
\item The  $ \ZZ$-graded covers $\wt{A}\operatorname{-Modf}^{\ZZ}$ and $\wh{A}\operatorname{-Modf}^{\ZZ}$
      of $A\operatorname{-Modf}$ are cover-equivalent;
\item There exists a $\ZZ$-grading on the abelian group $A$ 
making it a graded $\wh{A}$-$\wt{A}$-bimodule $\;{^{\wh{}}\!\!A^{\wt{}}}$;
\item  For each complete system of primitive pairwise orthogonal
  idempotents $( \Idp{x} )_{x \in I}$ in 
$A$, homogeneous  for the grading $\wh{A}$, 
  there exist a unit $u \in A^{\times}$ 
and a function $n :I\ra \ZZ$
  such that the homogeneous elements of $\wt{A}$ in degree $i$  
  are given by 
 \begin{center}
     $\wt{A}_i  =  \bigoplus_{x,y \in I} u \Idp{x} \wh{A}_{n(x) - n(y) + i} \Idp{y} u^{-1}$
  \end{center}
\item  There exist a complete system of primitive pairwise orthogonal
  idempotents $( \Idp{x} )_{x \in I}$ in 
$A$, homogeneous  for the grading $\wh{A}$, 
  a unit $u \in A^{\times}$ 
and a function $n :I\ra \ZZ$
  such that the homogeneous elements of $\wt{A}$ in degree $i$  
  are given by 
 \begin{center}
     $\wt{A}_i  =  \bigoplus_{x,y \in I} u \Idp{x} \wh{A}_{n(x) - n(y) + i} \Idp{y} u^{-1}$
  \end{center}      
          
\end{enumerate}

\end{prop}
\begin{remarkl}
  The conditions (3) and (4) were
needed for the proof of the main result in  an older version.
In this version, their only role is  making 
our conditions somewhat more explicit, but they are no more 
needed later on.
\end{remarkl}

\begin{proof}

(1) $\Leftrightarrow$ (2) follows from \ref{ggr}.
Next we prove (2) $\Rightarrow$ (3).
 Our graded bimodule  $\;{^{\wh{}}\!\!A^{\wt{}}}$ 
with $\op{id}: \hat v \;{^{\wh{}}\!\!A^{\wt{}}}\sira A$ is
a $\ZZ$-graded lift in
$\wh{A}\operatorname{-Modf}^{\ZZ}$ of the left $A$-module $A$.
By \ref{ex:ClassifyingLiftsOfA}, 
for each complete system of 
pairwise orthogonal idempotents $1_x\in A$, homogeneous for $\hat A$,
there exist integers $n(x)$ along  with an isomorphism
$$\psi:  \;{^{\wh{}}\!\!A^{\wt{}}}\;\sira \;\textstyle \bigoplus_{x \in
    I}\hat{A}\Idp{x}[n(x)]$$ of graded left $\hat A$-modules. 
Here both sides, when considered as ungraded left $A$-modules,
 admit obvious
 natural isomorphisms to the left $A$-module $A$.
In these terms
$\psi$ has to correspond to the right multiplication by a unit $u\in
A^\times$. 
Now certainly $h\mapsto \psi h\psi^{-1}$ 
is an isomorphism between the endomorphism rings of
these graded left $\wh{A}$-modules and 
with $a\mapsto u^{-1}au$ in the lower horizontal 
we  get a commutative diagram
$$
\begin{array}{ccc}
\op{End}_{A}\left( \;{^{\wh{}}\!\!A^{\wt{}}}\;\!\right)&\sira
&\op{End}_{A}\left(\textstyle \bigoplus_{x \in
    I}\hat{A}\Idp{x}[n(x)]\right)\\
\ua\wr&&\wr\ua\\
A&\sira&A
\end{array}$$ 
Here $\op{End}_{A}$ means  endomorphism rings of ungraded modules,
but with 
the grading coming from the grading on our modules, and 
the vertical arrows are meant to map $a\in A$ to  the multiplication 
by $a$ from the right modulo the obvious natural isomorphisms
mentioned above. In particular, the vertical maps are not compatible 
but rather ``anticompatible'' with the
multiplication. Nevertheless, the lower horizontal has
to be homogeneous for the gradings induced
by the vertical isomorphisms  from the upper horizontal
and from that we deduce
$$u^{-1}\wt A_iu=\textstyle  \bigoplus_{x,y}  1_x\wh A_{n(x)-n(y)+i}1_y$$
To prove (3) $\Rightarrow$ (4),
just recall that
by \ref{ex:ClassifyingLiftsOfA} we can 
always
find a complete system of primitive pairwise orthogonal
  idempotents $( \Idp{x} )_{x \in I}$ in $A$, which are 
homogeneous for the grading
  $\wh{A}$.
To finally check (4) $\Rightarrow$ (2),
just equip $A$ with the grading $\;{^{\wh{}}\!\!A^{\wt{}}}\;\!$
for which the right multiplication by $u$ as a map $(\cdot u):
\;{^{\wh{}}\!\!A^{\wt{}}}\;\!\sira \bigoplus_x \hat A 1_x[-n(x)]$ 
is homogeneous of 
degree zero. 
\end{proof}

\section{Gradings and bicentralizing modules}\label{gbz} 
\begin{defi}
 Let $A$ be a ring. An $A$-module $Q$ is called {\bf bicentralizing}
 if and only if the obvious map is an
isomorphism
\begin{equation*}
 A \overset{\sim}{\rightarrow} \op{End}_{\op{End}_A Q} Q
\end{equation*}
\end{defi}
\begin{remarkl}
For the artinian rings describing blocks of category $\mathcal O$,
the modules $Q$  corresponding to the antidominant prinjective
are bicentralizing. Indeed, the Struktursatz \cite{So} tells
us in this case, that the functor
\begin{equation*}
 \op{Hom}_A (\;, Q) : A\op{-Modf}^{\op{opp}} \rightarrow (\op{End}_A Q) \op{-Modf}
\end{equation*}
is fully faithful on projectives. On the other hand
it maps $A$ to $Q$ and thus induces an isomorphism\label{BCRT} 
$
 (\op{End}_A A)^{\op{opp}} \overset{\sim}{\rightarrow} \op{End}_{\op{End}_A Q} Q.
$
I learned this from \cite{CSX}.
\end{remarkl}
\begin{defi}
 Let $A$ be a ring and $Q \in A\op{-Mod}$ a bicentralizing $A$-module. 
Then we call a $\mathbb Z$-grading on $A$
and a $\mathbb Z$-grading on $\op{End}_A Q$ {\bf compatible} 
if and only if there exists a $\mathbb Z$-grading on
$Q$ compatible with both of them.
\end{defi}
\begin{prop}[{\bf Compatibility implies cover-equivalence}] 
 Let $A$ be a left-artinian ring and $Q$ a 
bicentralizing $A$-module. Then any two $\mathbb Z$-gradings on $A$,
which are compatible with the same 
$\mathbb Z$-grading on $\op{End}_A Q$,\label{cG}  
give rise to  cover-equivalent covers of $A\op{-Modf}$.
\end{prop}
\begin{proof}
 Let $\tilde A$ and $\hat A$ be our two $\mathbb Z$-gradings on 
$A$. By assumption, there exist $\mathbb Z$-gradings 
$\tilde Q$ and $\hat Q$ on $Q$
compatible
with the given grading on
$\op{End}_A Q$ and  compatible
with the gradings $\tilde A$ and $\hat A$ of $A$ respectively.
But then let us consider the isomorphism
\begin{equation*}
 A \overset{\sim}{\rightarrow} \op{Hom}_{\op{End}_A Q} (\tilde Q, \hat Q)
\end{equation*}
given by the left action of $A$ on $Q$ and denote by $\;{^{\wh{}}\!\!A^{\wt{}}}$  the group $A$ with
the $\mathbb Z$-grading coming from the right hand side of this isomorphism by transport of
structure.
Then $\;{^{\wh{}}\!\!A^{\wt{}}}$ with its obvious left and right action is a $\mathbb Z$-graded $\hat A$-$\tilde A$-bimodule
and the proof is finished by \ref{prop:EquivalenceCriterium}.
\end{proof}
\begin{prop}[{\bf Compatibility criterion}] 
 Let $A$ be a left-artinian ring and let $Q$ be a projective finite
   length bicentralizing left\label{CCE} 
$A$-module with commutative endomorphism ring.
Then the compatibility of a grading on $A$  with a  grading on $\op{End}_A
Q$ is equivalent to the homogeneity of the composition
\begin{equation*}
 \op{End}_A Q \rightarrow \op{End}_{\op{End}_A Q} Q \overset{\sim}{\leftarrow} A
\end{equation*}
\end{prop}
\begin{proof}
 If the gradings are compatible, clearly this composition is homogeneous.
If on the other hand the composition is homogenous, 
then any grading on the $A$-module $Q$ will show
the compatibility.
However such a grading always exists by \ref{prop:ProjectivesLift}.
\end{proof}

\section{Cover equivalence is an equivalence relation} 

\begin{remarkl}
  Clearly cover-equivalence of graded covers is a reflexive relation.
We are now going to\label{coeq} 
show it is also symmetric and 
transitive,
so it is indeed  an equivalence relation on the set of
graded covers of a fixed artinian
category.
\end{remarkl} 
\begin{lemma}
  Any graded lift $\tilde F$ 
of an equivalence $F$  of
artinian categories is again an equivalence of categories.
\end{lemma}
\begin{proof}
 Since a direct sum of morphisms of abelian groups is an isomorphism if and
only if the individual morphisms are isomorphisms, a graded lift of a
fully faithful additive functor is clearly fully faithful itself.
We just have to show that if $F$ is an equivalence of categories, 
then
$\tilde F$ is essentially surjective. 
So take an object $\tilde B\in\tilde{\mathcal B}$.
  By assumption there is an object $A\in \mathcal A$ with an isomorphism
$FA\sira w\tilde B$. By the definition of graded cover and \ref{ACO},
there are $\tilde X,\tilde Y\in\tilde{\mathcal A}$ with
an epimorphism and a monomorphism 
$v\tilde X\sra A\hra v\tilde Y$. Applying $F$ we find
an epimorphism and a monomorphism 
$w\tilde F\tilde X\sra w\tilde B\hra w\tilde F\tilde Y$. 
Now if $\lambda_i:\tilde F\tilde X[i]\ra \tilde B$ for $i$
runnig through a finite set $I\subset\DZ$ of degrees are
the homogeneous components of the first map, then
they together define the left  epimorphism of a sequence
$$\bigoplus_{i\in I}\tilde F\tilde X[i]\sra \tilde B\hra \bigoplus_{j\in J}\tilde F\tilde Y[j]$$
 in $\tilde{\mathcal B}$. The left monomorphism is constructed dually.
But the composition in this sequence has to come from
a morphism in $\tilde{\mathcal A}$, and 
the image of this morphism is the looked-for object of 
 $\tilde{\mathcal A}$ essentially going to $\tilde{B}$
under our functor $\tilde F$. 
\end{proof}
\begin{remarkl}[{\bf Symmetry of cover-equivalence}] 
 In particular,
given a cover-equivalence $(F,\pi)$ of graded covers  the functor $F$ is 
always an
equivalence of categories. 
Given a quasiinverse $(G,\eta)$ with
$\eta:\op{Id}\stackrel{\sim}{\RA} FG $ an isotransformation,
from $\pi: \hat v F\stackrel{\sim}{\RA} \tilde v$ we get
as the composition $\hat v\stackrel{\sim}{\RA}  \hat v FG
\stackrel{\sim}{\RA} \tilde v G$ or more precisely 
 $(\pi G)(\hat v\eta)$ an isotransformation 
$\tau: \hat v\stackrel{\sim}{\RA}  \tilde v G$.
Similarly from $\epsilon:[1] F\stackrel{\sim}{\RA} F[1]$
we get a unique $\epsilon:[1] G\stackrel{\sim}{\RA} G[1]$
such the composition 
$$ [1]\stackrel{\sim}{\RA} [1]FG\stackrel{\sim}{\RA}F[1]G
\stackrel{\sim}{\RA}FG[1]\stackrel{\sim}{\RA}[1]$$
with our adjointness $\eta$ at both ends and the old and the newly to be defined
$\varepsilon$ in the middle ist the identity transformation. 
Then one may check that
$(G,\tau,\epsilon)$ is also a cover-equivalence. 
\end{remarkl}

\begin{remarkl}[{\bf Transitivity of cover-equivalence}]
 Let finally be given artinian categories $\calA, \calB, \calC$,
  additive functors
$F: \calA \rightarrow \calB$ and $G: \calB \rightarrow \calC$, 
graded covers $(\wt{\calA}, v, \eps), (\wt{\calB},  w , \eta),
(\wt{\calC}, u, \theta)$ of
$\calA,\calB,\calC$, and graded lifts $(\tilde F,\pi,\varepsilon)$ of $F$ and
$(\tilde G,\tau,\varepsilon)$ of  $G$. 
Then $(\tilde G\tilde F, (G\pi)(\tau\tilde F),\varepsilon)$ is a lift of $GF$,
where we leave the definition of the last $\varepsilon$ to the reader. 
In particular the relation of cover-equivalence  is transitive.  
\end{remarkl}

\section{Proof of the Main Theorem}
\begin{remarkl}[{\bf Push-forward of graded covers}] 
Let $(\wt{\mathcal A},\td v, \td \eps)$ 
be a\label{qezt} 
  graded cover of an artinian category $\mathcal A$ and
let $ E:\mathcal A\sirra \mathcal B$ be an equivalence of categories.
Then obviously 
$(\wt{\mathcal A}, E \td v,  E(\td \eps))$ is a graded cover of
$\mathcal B$. We call it the ``push-forward'' 
of our graded cover of $\mathcal A$
along $ E$. 
Obviously
two graded covers of $\mathcal A$ are cover-equivalent 
iff their pushforwards are
cover-equivalent as graded covers of $\mathcal B$. 
\end{remarkl}

\begin{remarkl}
Recall the block decomposition of $\mathcal O$. 
  It is enough to prove theorem \ref{thm:MainTheorem} for each block
  $\scO_{\lambda}$ of $\scO$.
\end{remarkl}

\begin{thm}[{\bf Uniqueness of graded covers of $\scO$}]
If two graded covers 
$(\wt{\scO}_{\lambda}, \wt{v}, \wt{\eps})$ and\label{thm: GradingsOnO}
$(\wh{\scO}_{\lambda}, \wh{v}, \wh{\eps})$ of 
a block $\scO_{\lambda}$ of category $\scO$ are both compatible 
with the action of the center, they are cover-equivalent.
\end{thm}

\begin{proof}
  We find a left artinian ring $A$ 
along with an equivalence $E:\mathcal O_\lambda\sira A\op{-Modf}$. 
It is sufficient to show that the pushforwards of our covers along
$E$ are cover-equivalent. Both these pushforwards 
are  cover-equivalent to 
covers corresponding to $\DZ$-gradings $\hat A$ and $\tilde A$ on
$A$, which are compatible with the action of the center  in the sense
that the maps $Z/\chi^n\ra \hat A$ and $Z/\chi^n\ra \tilde A$
are homogeneous for $\chi\in\op{Max}Z$ the corresponding central character
and $n$ so big, that our maps are well defined. 
But now the antidominant prinjective of our block corresponds to a 
bicentralizing $A$-module $Q$ by \ref{BCRT}, and the 
action of the center induces a surjection $Z/\chi^n\sra \op{End}_AQ$
by the Endomorphismensatz of \cite{So}, and by the compatibility 
criterion \ref{CCE} both our $\DZ$-gradings on $A$ are compatible with the
same $\DZ$-grading on $\op{End}_AQ$. Then however 
the corresponding covers are cover-equivalent by \ref{cG}.
\end{proof}

\end{document}